\documentclass[reqno, english]{amsart}
\usepackage{etex}
\usepackage{amsmath,amssymb,amsthm,bbm,mathtools,comment}
\usepackage[shortlabels]{enumitem}
\usepackage[pdftex,colorlinks,backref=page,citecolor=blue]{hyperref}
\hypersetup{pdfpagemode=UseNone,pdfstartview={XYZ null null 1.00}}
\usepackage[mathscr]{euscript}
\usepackage[usenames,dvipsnames]{color}
\usepackage{adjustbox,tikz,calc,graphics,babel,standalone}
\usetikzlibrary{shapes.misc,calc,intersections,patterns,decorations.pathreplacing}
\usetikzlibrary{arrows,shapes,positioning}
\usetikzlibrary{decorations.markings}
\usepackage[final]{microtype}
\usepackage[numbers]{natbib}
\usepackage{cmtiup}
\usepackage{amsfonts}
\usepackage{graphicx}
\usepackage{caption}
\usepackage{subcaption}
\usepackage{verbatim}
\usepackage{array}
\usepackage[frame,cmtip,arrow,matrix,line,graph,curve]{xy}
\usepackage{graphpap, color, pstricks}
\usepackage{pifont}
\usepackage[final]{microtype}
\usepackage{cmtiup}
\usepackage{todonotes}
\usetikzlibrary{topaths,calc} 
\tikzstyle{vertex} = [fill,shape=circle,node distance=80pt]
\tikzstyle{edge} = [opacity=0.4,fill opacity=0.0,line cap=round, line join=round, line width=40pt]
\tikzstyle{elabel} =  [fill,shape=circle,node distance=30pt]

\allowdisplaybreaks

\theoremstyle{plain}
\newtheorem{theorem}{Theorem}[section]		
\newtheorem{lemma}[theorem]{Lemma}
\newtheorem{claim}[theorem]{Claim}

\newtheorem{conjecture}[theorem]{Conjecture}

\newtheorem{observation}[theorem]{Observation}

\theoremstyle{remark}

\let\originalleft\left
\let\originalright\right
\renewcommand{\left}{\mathopen{}\mathclose\bgroup\originalleft}
\renewcommand{\right}{\aftergroup\egroup\originalright}

\makeatletter
\def\imod#1{\allowbreak\mkern10mu({\operator@font mod}\,\,#1)}
\makeatother

\title{Immersions of directed graphs in tournaments}
\author{Ant\'onio Gir\~ao}
\author{Robert Hancock}

\thanks{
AG: Mathematical Institute, University of Oxford, Oxford OX2 6GG, UK. E-mail: {\tt girao@maths.ox.ac.uk}. {Research supported by EPSRC grant EP/V007327/1}}
\thanks{
RH: Mathematical Institute, University of Oxford, Oxford OX2 6GG, UK. E-mail: {\tt robert.hancock@maths.ox.ac.uk}. Research supported by ERC Advanced Grant 883810 and by a Humboldt Research Fellowship at Heidelberg University.}

\begin{document}
\maketitle
\begin{abstract}
Recently, Dragani\'c, Munh\'a Correia, Sudakov and Yuster~\cite{Draganicetal} showed that every tournament on $(2+o(1))k^2$ vertices contains a $1$-subdivision of a transitive tournament on $k$ vertices, which is tight up to a constant factor. 
We prove a counterpart of their result for immersions. 
Let $f(k)$ be the smallest integer such that any tournament on at least $f(k)$ vertices must contain a $1$-immersion of a transitive tournament on $k$ vertices. We show that $f(k)=O(k)$, which is clearly tight up to a multiplicative factor. 
If one insists in finding an immersion of a complete directed graph on $k$ vertices then an extra condition on the tournament is necessary. 
Indeed, we show that every tournament with minimum out-degree at least $Ck$ must contain a $2$-immersion of a complete digraph on $k$ vertices. This is again tight up to the value of $C$ and tight on the order of the paths in the immersion. 
\end{abstract}

\section{Introduction}
Given a graph $G$, an \emph{immersion} of $G$ is any graph which can be obtained by replacing the edges of $G$ by pairwise edge-disjoint paths of any length. This is a relaxation of the notion of $\textit{subdivision}$ since one does not require the paths to be vertex-disjoint. Moreover, a \emph{$1$-immersion} is an immersion where all paths have length at most $2$. There is a lot of literature on extremal problems regarding subdivisions and immersions of graphs and digraphs. 
Indeed, a classical result of Bollob\'as and Thomason~\cite{BollobasThomason} independently proved by K\'omlos and Szemer\'edi~\cite{KomlosSzemeredi} states that any graph with average degree at least $Ct^2$ contains a subdivision of a complete graph on $t$ vertices. More recently DeVos, Dvo\v{r}\'ak, Fox, McDonald, Mohar and Scheide~\cite{Foxetal} showed that average degree at least $200t$ is enough to guarantee an immersion of a clique on $t$ vertices. It is not unlikely though that a lower bound of $(1+o(1))t$ for the minimum degree of a graph is enough to guarantee an immersion of $K_t$. Liu, Wang and Yang~\cite{hongetal} confirmed this provided the graph does not contain a fixed complete bipartite graph.

It is therefore natural to ask what happens in the directed setting. A complete digraph on $t$ vertices is a digraph where every pair of vertices are joined by two edges, one in each direction.
The first author together with Popielarz and Snyder~\cite{Popielarz} showed that any tournament with minimum out-degree at least $(2+o(1))t^2$ must contain a subdivision of a complete digraph on $t$ vertices and it follows from a construction of Thomassen~\cite{Thomasseneven} that an analogous result for arbitrary digraphs cannot hold, even to find an immersion: for every $k$, there is a digraph with minimum out and in-degree $k$ which does not contain an immersion of a complete digraph on at least $3$ vertices. One could then ask a slightly weaker question. Is there a function $f: \mathbb{N}\rightarrow \mathbb{N}$
 such that every digraph with minimum out-degree at least $f(k)$ contains a subdivision of a transitive tournament on $k$ vertices? A positive answer was conjectured to be true by Mader~\cite{Mader3} but unfortunately the answer remains elusive. We note however that Lochet~\cite{Lochet} showed in a beautiful paper that Mader's Conjecture does hold for the immersions of transitive tournaments. 
 
 Erd\H{o}s and Burr~\cite{burr1975magnitude} conjectured that subdivisions of graphs
in which each subdivision path is of length at least $2$, have a Ramsey number which is linear in the number
of vertices. Alon~\cite{Alon} resolved this in $1994$, showing that every graph on $n$ vertices in which no two vertices
of degree at least $3$ are adjacent has Ramsey number at most $Cn$. Later, Alon, Krivelevich and Sudakov~\cite{AlonKrivSudak} proved a stronger density-type result for complete graphs, showing that every $n$-vertex graph with at least $\varepsilon n^2$ edges must contain the $1$-subdivision of a complete graph on $c(\varepsilon) \sqrt{n}$ vertices which proves an old conjecture of
Erd\H{o}s. Very recently, Dragani\'c, Munh\'a Correia, Sudakov and Yuster~\cite{Draganicetal} proved an analogous Ramsey-type result in the tournament setting. Namely, that every tournament on at least $(2+o(1))k^2$ vertices must contain a $1$-subdivision of a transitive tournament on $k$ vertices, which is tight up to a factor of $4$. In our main result, we prove a similar statement for immersions. 

We recall a couple of definitions first. 
As is the case with subdivisions, for an immersion $I$ of $G$, we call the original vertices of $G$ the \emph{branch} vertices of $I$.
A \emph{strong} immersion is an immersion where none of the internal vertices in paths between branch vertices are themselves branch vertices.
%Indeed, we show that every tournament on at least $Ck$ vertices must contain a $1$-immersion of a transitive tournament on $k$ vertices. 

\begin{theorem}\label{thm:immersion1}
There exists a constant $C>0$ such that all tournaments $T$ on $Ck$ vertices contain a strong $1$-immersion of a transitive tournament on $k$ vertices. 
\end{theorem}

We should point out that the techniques used in \cite{Draganicetal} do not work in the immersion setting since th
ey heavily rely on the fact that the number of branch vertices  is much less than the size of the whole tournament. 
Note as well that we have made no attempt to optimise the constant $C$ in Theorem~\ref{thm:immersion1}, 
since the use of Lemma~\ref{thm:prob} in the proof caused $C$ to be quite far away from the lower bound, which we now discuss. 

\begin{observation}
For all $k \geq 2$, there exists a tournament on $2k-3$ vertices which does not contain a immersion of a transitive tournament on $k$ vertices.
\end{observation}
An immersion of a transitive tournament on $k$ vertices must contain a vertex with out-degree at least $k-1$.
So take any regular tournament on $2k-3$ vertices, i.e. one where every vertex has in-degree and out-degree exactly $k-2$.

If one is interested in finding an immersion of a complete digraph on $k$ vertices, then extra conditions are necessary. Indeed,
take a transitive tournament on any number of vertices, and replace the tournament within the $2k-3$ vertices of smallest out-degree by a regular tournament.
The resulting tournament has minimum out-degree $k-2$ and does not contain the desired immersion: the final $2k-3$ vertices cannot play the role of a branch vertex since every branch vertex must have out-degree at least $k-1$, while for any pair of possible branch vertices elsewhere, there is no directed path from the vertex of smaller out-degree to the vertex of larger out-degree.
We show that this lower bound is essentially tight. 

\begin{theorem}\label{thm:immersion2}
There exists a constant $C>0$ such that any tournament $T$ with $\delta^{+}(T)\geq Ck$ 
contains a strong $2$-immersion of a complete directed graph on $k$ vertices. %where each path has at most $3$ edges. 
\end{theorem}
In particular, using our method, $C=59$ suffices. We aim to present a clear proof rather than improve the value of $C$ here. 
Finally, we observe that the order of the paths in the immersion is tight. Form $T$ by taking a blow-up of a directed triangle, with parts $U,V,W$ where $U\rightarrow V$, $V\rightarrow W$ and $W\rightarrow U$ and placing a transitive tournament in each of $U,V$ and $W$. 
Then, if $k\geq 4$, two branch vertices $x$ and $y$ must be within the same part, say $U$, and then either the path from $x$ to $y$ or the path from $y$ to $x$ in the immersion must have at least $3$ edges.

\subsection{Notation and organisation}\label{sec:notation}

Given a graph $G$, an \emph{immersion} $I$ of $G$ is any graph which can be obtained by replacing the edges of $G$ by pairwise edge-disjoint paths of any length. As is the case with subdivisions, we call the original vertices of $G$ in $I$ the \emph{branch} vertices of $I$.
Given a positive integer $t$, a \emph{$t$-immersion} is an immersion where all paths have length at most $t+1$.
A \emph{strong immersion} is an immersion where none of the internal vertices in paths between branch vertices are themselves branch vertices.

We use standard notation for directed graphs and tournaments: given a directed graph $G$ and vertex $x \in V(G)$, we write $N_G^{+}(x)$ ($N_G^{-}(x)$) for the out-neighbourhood (in-neighbourhood) of $x$ in $G$, and as usual omit the subscript $G$ if $G$ is clear from the context. We write $d_G^{+}(x):=|N_G^{+}(x)|$ and $d_G^{-}(x):=|N_G^{-}(x)|$.
We write $\delta^{+}(G)$ ($\delta^{-}(G)$) for the minimum out-degree (in-degree) in $G$. For two vertex sets $A,B \subseteq V(G)$, write $A \rightarrow B$ to say that every edge between $A$ and $B$ is directed from $A$ to $B$.

Throughout the paper, all logarithms are base $2$. Finally, let $[n]:=\{1,\dots,n\}$.

We prove Theorem~\ref{thm:immersion1} in Section~\ref{sec:proof1} and Theorem~\ref{thm:immersion2} in Section~\ref{sec:proof2},
and give some concluding remarks in Section~\ref{sec:conclusion}.

\section{Proof of Theorem~\ref{thm:immersion1}}\label{sec:proof1}

First we give a sketch of the proof. Order the vertices by decreasing outdegree $x_1,\dots,x_{Ck}$, pick a subset $S=\{x_{i_1},\dots,x_{i_k}\}$ of size $k$, and suppose we are trying to find an immersion using these $k$ vertices as branch vertices. Due to the ordering, the edges are more likely to be oriented `forwards', from $x_{i_t}$ to $x_{i_j}$ where $t<j$. We use as many forwards edges as we can for the immersion. For each backwards edge $e$, we must instead find a directed path of length $2$. Let $R_S(e)$ be the set of all possible middle vertices of these paths. We encode this information in a bipartite graph $G(S)$ with parts $X,Y$, with $X$ the backwards edges in $S$, and $Y$ the set of possible middle vertices of paths. We match $x \in X$ to $y \in Y$ if $y \in R_S(x)$. Now, a matching which covers $X$ precisely corresponds to a selection of edges which gives us the desired $1$-immersion. 

We show that there exists a choice of $S$ for which this matching exists: we choose a set randomly, and define vertices to be good or bad (with respect to this random choice) in such a way that, if one selects a subset of size $k$ where all vertices are good, then the matching can be found. We then show that with positive probability that there is a subset of size $k$ where all vertices are good. Since the initial set was chosen randomly, this implies the existence of the desired set $S$.

We now proceed with the proof. First we need a simple Hall-type lemma.

\begin{lemma}\label{lem:pairs}
Let $G=G[X,Y]$ be a bipartite graph and $\mathcal{H}$ a hypergraph with vertex set $X$. 
Suppose no vertex in $X$ is isolated in $G$.
Suppose that for every hyperedge $F\in \mathcal{H}$, and every positive $i\leq \log(|F|)$, 
there are at most $2^{i-1}$ vertices in $V(F)$ with degree in $G$ at most $2^{i+1}$. 
Further, suppose that every $v \in X$ lies in at most two hyperedges of $\mathcal{H}$.
Then we can find a pairing $\mathcal{P}$ from $X$ to $Y$ such that for every $F\in \mathcal{H}$ the pairing $\mathcal{P}$ restricted to $F$ forms a matching. 
\end{lemma}

\begin{proof}
Let $V_1:=\{v \in X : d_G(v) \leq 4\}$ and
for each positive $2 \leq i \leq \log(|F|)$, 
let $V_i:=\{v \in X :  2^{i} < d_G(v) \leq 2^{i+1}\}$.
We choose the pairing from $X$ to $Y$ greedily, going through the sets $V_1,V_2,\dots$ in turn.

First for $i=1$, we greedily pick any neighbour in $Y$ for each vertex in $V_1$;
since no $F$ contains more than $1$ vertex from $V_1$, the pairing so far restricted to any $F$ is a matching.

Now let $i \geq 2$. Let $v \in V_i$.
Let $F,F'$ be the at most two hyperedges which $v$ lies in. 
Each $F$ and $F'$ contain at most $2^{i-1}$ vertices with degree in $G$ at most $2^{i+1}$.
Therefore there are at most $2^i$ vertices which already have partners in $Y$ and cannot have the same partner as $v_i$.
Since $d_G(v_i)>2^i$, there is at least one choice of partner in $Y$ so that the pairing restricted to any hyperedge is still a matching. 
Since $i$ and $v_i$ were arbitrary, this completes the proof.
\end{proof}

\begin{proof}[Proof of Theorem~\ref{thm:immersion1}]
Let $C$ be a constant to be chosen later. 
Let $T$ be a tournament on $Ck$ vertices.
Let $x_1,\ldots ,x_{Ck}$ be an ordering $<$ of the vertices of $T$ by decreasing out-degree in $T$ i.e. $d^{+}(x_i) \geq d^{+}(x_{i+1})$. 
For each $x_i \in V(T)$, let $A^{-}(x_i):=N^{-}(x_i) \cap \{x_j : j>i\}$ and $B^{+}(x_i):=N^{+}(x_i) \cap \{x_j : j<i\}$. 
(Morally $A$ stands for ahead and $B$ stands for behind.)
For every $x\in V(T)$, we define an ordering $<_x^{-}$ of $A^{-}(x)$ by increasing out-degree within $A^{-}(x)$ i.e. $u<w$ if $d_{A^{-}(x)}^{+}(u)\leq d_{A^{-}(x)}^{+}(w)$. 
Similarly we define an ordering $<_x^{+}$ of $B^{+}(x)$ by increasing in-degree within $B^{+}(x)$.
Given $\ell\in [|A^{-}(x)|]$, let $J^{\ell}(A^{-}(x))$ be the interval of the first $\ell$ vertices of $A^{-}(x)$ in $<_x^{-}$.
Define $J^{\ell}(B^{+}(x))$ analogously.

Given a fixed subset $S\subset V(T)$ of the vertices of $T$, say $x_{i_1},x_{i_2},\ldots, x_{i_k}$ (with $i_1<i_2<\cdots<i_k$), 
we define $D(S)$ to be the digraph with vertex set $S$ together with the backward edges of $T$, i.e. $x_{i_j} x_{i_t}$ where $j>t$. 
For each such backwards edge $e=x_{i_j}x_{i_t}$ we define
\begin{align}\label{eq:Rwhichworks}
R_S(e):= \{z \in V(T) \setminus S : x_{i_t}z, zx_{i_j} \in E(T) \}.
\end{align} 

Let $G(S)$ be a bipartite graph with $X:=E(D(S))$, $Y:=V(T)\setminus S$, and for each $a \in X, b \in Y$, we have $ab \in E(G(S))$ if $b \in R_S(a)$. 
Further, for each $x \in S$, let $F^-_x$ be the set of all in-edges of $x$ in $X$, and $F^+_x$ the set of all out-edges of $x$ in $X$.
Let $\mathcal{H}(S)$ be a hypergraph with vertex set $X$, and edge set $\{F^-_x,F^+_x : x \in S\}$. 
Now observe that if the pair $(G(S),\mathcal{H}(S))$ satisfies the conditions of Lemma~\ref{lem:pairs}, then we can find a strong $1$-immersion $I$ in $T$ which has branch vertex set $S$.
For each pair of vertices $x_{i_t}, x_{i_j} \in S$ with $t<j$, 
if $x_{i_t} x_{i_j} \in E(T)$ then we include $x_{i_t} x_{i_j}$ in $I$. 
If not, then put the directed path of length $2$ from $x_{i_t}$ to $x_{i_j}$ which goes via the chosen representative $s_{x_{i_j}x_{i_t}} \in R_S(x_{i_j}x_{i_t})$ in $I$.
Observe that the conclusion of Lemma~\ref{lem:pairs} precisely ensures that all paths of length $2$ in $I$ are pairwise edge-disjoint. 
In particular, it ensures that if $z \in R_S(x_{i_j}x_{i_t})$ is chosen for the edge $x_{i_j}x_{i_t}$, then there is no other $x_{i_q}$ for which $z$ is chosen for the edge $x_{i_j}x_{i_q}$ or the edge $x_{i_q}x_{i_t}$.
Further, the definition of $R_S(e)$ ensures that no path of length $2$ in $I$ contains a branch vertex as the middle vertex of the path.
As well as implying that all paths of length $2$ are pairwise edge-disjoint from all paths of length $1$ (which in turn ensures that $I$ is indeed an $1$-immersion),  
this implies that $I$ is a strong $1$-immersion. 

It now suffices to show that we can find an $S$ as above of size at least $k$.
To this end, we will take each vertex of $V(T)$ independently with probability $p$ (also a constant to be chosen later). 
Let $S$ be the random set obtained. 
We will show that with positive probability $S$ contains a subset $S'$ of size at least $k$ for which the pair $(G(S'),\mathcal{H}(S'))$ satisfies the conditions of Lemma~\ref{lem:pairs}. 

Given a vertex $x \in V(T)$, we say it is \emph{in-bad} if: 
\begin{enumerate}
\item $x\in S $, and either;
\item there exists $i \in [|A^{-}(x)|]$ such that $|S\cap J^{i}(A^{-}(x))| > i/10 $, or;
\item there exists $i \in [|A^{-}(x)|]$ such that, letting $z_i$ be the $i$-th vertex in the ordering $<_x^{-}$ of $A^{-}(x)$, then at least $i/20$ of the first $\lfloor i/2 \rfloor$ vertices in $N^{+}(x) \cap N^{-}(z_i)$ are in $S$. 
\end{enumerate}
Symmetrically, we say that $x$ is \emph{out-bad} if:
\begin{enumerate}
\item $x\in S$, and either;
\item there exists $i \in [|B^{+}(x)|]$ such that $|S\cap J^{i}(B^{+}(x))| > i/10 $, or;
\item there exists $i \in [|B^{+}(x)|]$ such that, letting $z_i$ be the $i$-th vertex in the ordering $<_x^{+}$ of $B^{+}(x)$, then at least $i/20$ of the first $\lfloor i/2 \rfloor$ vertices in $N^{-}(x) \cap N^{+}(z_i)$ are in $S$. 
\end{enumerate}
 
We now make two claims which together will quickly imply the result.

\begin{claim}\label{cl:1}
Let $S$ be a subset of $V(T)$ such that no vertex in $S$ is in-bad or out-bad. 
Then $(G(S),\mathcal{H}(S))$ satisfies the conditions of Lemma~\ref{lem:pairs}. 
\end{claim}

\begin{claim}\label{cl:2}
The probability $x$ is in-bad or out-bad is at most 
$$c:=2p \left( \sum_{i=1}^{\infty} \left( (10ep)^{i/10} + (10ep)^{i/20} \right) \right).$$ 
\end{claim}

Let $S' \subseteq S$ be such that all $x \in S'$ are not in or out-bad. 
We have 
\begin{align*}
\mathbb{E}[|S'|] & \geq \mathbb{E} [|S|] - \mathbb{E} [|\{x \in V(T): x \text{ in-bad or out-bad}\}|] \\
& \geq Ck (p-c) \geq k,
\end{align*}
where the last inequality follows by suitable choices of $p$ and $C$; 
e.g. with $p=\frac{1}{10e \cdot 4^{20}}$ we obtain $c = \frac{4p}{5}$ and so $C \geq \frac{5}{p} = 50e \cdot 4^{20}$ suffices. 

It remains to prove each of the claims.

\begin{proof}[Proof of Claim~\ref{cl:1}]
It is clear that every $a \in X$ (recall that $a$ is a backwards edge of $T$) lies in at most two hyperedges of $\mathcal{H}(S)$. 
Now let $x\in S$ and consider $F^-_x$, the set of in-edges of $x$ in $X$.
Let $Z=\{z_1,\dots,z_{|F^-_x|}\}$ be the other ends of these in-edges, i.e. the in-neighbours of $x$ in $S$, 
ordered by $<_x^{-}$.

Fix $j \in [|F^-_x|]$.
Choose the $i \in [|A^{-}(x)|]$ such that the $i$-th vertex in the ordering $<_x^{-}$ of $A^-(x)$ is $z_j$. 
Since $x$ is not in-bad, by (2) we have that $j \leq  i/10$ (and $i\geq 10$).

Recall that any tournament on $i$ vertices contains a vertex of out-degree at least $\lfloor i/2 \rfloor$,
in particular this is true of the subtournament $J^i(A^{-}(x))$.
Since $z_j$ has out-degree in $A^{-}(x)$ at least as large as all the other vertices in $J^i(A^{-}(x))$, we therefore have $d_{A^{-}(x)}^{+}(z_j) \geq \lfloor i/2 \rfloor$.
Also observe that all vertices $y \in N_{A^{-}(x)}^{+}(z_j)$ satisfy that $z_jy$ and $yx$ are edges of $T$, 
and so $|N^+(z_j) \setminus N^+(x)| \geq \lfloor i/2 \rfloor$.
Since $x<z_j$ in the original ordering $<$,  we also must have $|N^+(x) \setminus N^+(z_j)| \geq \lfloor i/2 \rfloor$. 
Finally observe by (3) that at most $\lfloor i/20 \rfloor$ of the first $\lfloor i/2 \rfloor$ of these vertices are in $S$, in other words 
we have $|R(z_jx)| \geq \lfloor i/2 \rfloor - \lfloor i/20 \rfloor \geq 4i/10 \geq 4j$. 
(See Figure~\ref{fig:1} for an example.)

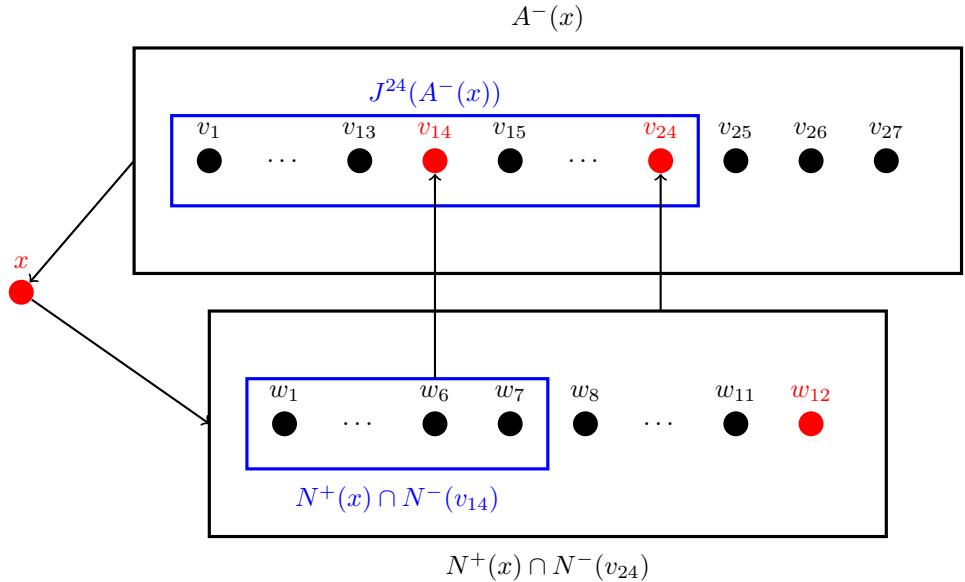
\begin{figure}[h]
\begin{center}
\begin{tikzpicture}[scale=1.0]

%\filldraw[color=black!100, fill=red!0, very thick](7,3) ellipse (5.5 and 1.6);
\filldraw[color=black!100, fill=red!0, very thick](1.5,1.5) rectangle (12.5,4.5);
\node at (7,4.9) {$A^{-}(x)$};

\filldraw[color=blue!100, fill=red!0, very thick](2,2.4) rectangle (9,3.6);
\node at (5.5,3.9) {{\color{blue}$J^{24}(A^{-}(x))$}};

\node[vertex,color=red!100] at (0,1.25) (x) {};
\node at (0,1.65) {{\color{red}$x$}};

\node[vertex] at (2.5,3) (x1) {};
\node[vertex] at (4.5,3) (x13) {};
\node[vertex,color=red!100] at (5.5,3) (x14) {};
\node[vertex] at (6.5,3) (x15) {};
\node[vertex,color=red!100] at (8.5,3) (x24) {};
\node at ($(x1)!.5!(x13)$) {\ldots};
\node at ($(x15)!.5!(x24)$) {\ldots};
\node[vertex] at (9.5,3) (x25) {};
\node[vertex] at (10.5,3) (x26) {};
\node[vertex] at (11.5,3) (x27) {};
\node at (2.5,3.4) {$v_1$};
\node at (4.5,3.4) {$v_{13}$};
\node at (5.5,3.4) {{\color{red}$v_{14}$}};
\node at (6.5,3.4) {$v_{15}$};
\node at (8.5,3.4) {{\color{red}$v_{24}$}};
\node at (9.5,3.4) {$v_{25}$};
\node at (10.5,3.4) {$v_{26}$};
\node at (11.5,3.4) {$v_{27}$};

\filldraw[color=black!100, fill=red!0, very thick](2.5,-2) rectangle (11.5,1);
\node at (7,-2.4) {$N^{+}(x) \cap N^{-}(v_{24})$};

\filldraw[color=blue!100, fill=red!0, very thick](3,-1.1) rectangle (7,0.1);
\node at (5,-1.5) {{\color{blue}$N^{+}(x) \cap N^{-}(v_{14})$}};

\node[vertex] at (3.5,-0.5) (w1) {};
\node[vertex] at (5.5,-0.5) (w6) {};
\node[vertex] at (6.5,-0.5) (w7) {};
\node[vertex] at (7.5,-0.5) (w8) {};
\node[vertex] at (9.5,-0.5) (w11) {};
\node at ($(w1)!.5!(w6)$) {\ldots};
\node at ($(w8)!.5!(w11)$) {\ldots};
\node[vertex,color=red!100] at (10.5,-0.5) (w12) {};
\node at (3.5,-0.1) {$w_1$};
\node at (5.5,-0.1) {$w_{6}$};
\node at (6.5,-0.1) {$w_{7}$};
\node at (7.5,-0.1) {$w_{8}$};
\node at (9.5,-0.1) {$w_{11}$};
\node at (10.5,-0.1) {{\color{red}$w_{12}$}};

\draw[<-,thick](x) -- (1.5,3);
\draw[->,thick](x) -- (2.5,-0.5);
\draw[<-,thick](x14) -- (5.5,0.1);
\draw[<-,thick](x24) -- (8.5,1);
\end{tikzpicture}
\end{center}
\caption{An example where $x$ is in $S$ and $x$ is not in-bad. Vertices in red ($x,v_{14},v_{24},w_{12}$) are in $S$.  The vertices $v_1,\dots,v_{27}$ are ordered by $<^{-}_{x}$, i.e. we have $d^{+}_{A^{-}(x)}(v_i) \geq d^{+}_{A^{-}(x)}(v_j)$ for all $j<i$. In particular, this guarantees $d^{+}_{A^{-}(x)}(v_{14}) \geq 7$ and $d^{+}_{A^{-}(x)}(v_{24}) \geq 12$. In turn, we see $\{w_1,\dots,w_7\} \in N^{+}(x) \cap N^{-}(v_{14})$ and $\{w_1,\dots,w_{12}\} \in N^{+}(x) \cap N^{-}(v_{24})$. Relabelling $z_1:=v_{14}$ and $z_2:=v_{24}$, we see $|R(z_1 x)|=7 \geq 4$ and since $R(z_2 x) = (N^{+}(x) \cap N^{-}(v_{24})) \setminus S$, we have $|R(z_2 x)|=11 \geq 8$.} 
\label{fig:1}
\end{figure}

First this implies that $X$ has no isolated vertices in $G(S)$. 
Further, since $j \in [|F^-_x|]$ was arbitrary, 
it follows that the hyperedge $F^-_x$ contains at most $2^{i-1}$ vertices with degree in $G(S)$ at most $2^{i+1}$ for each $i$, as required.  

For $F^+_x$ the set of out-edges of $x$ in $D(S)$, a symmetrical argument to the above gives the same conclusion for $F^+_x$.
\end{proof}

\begin{proof}[Proof of Claim~\ref{cl:2}]
We will use the following bound on expectation.
\begin{lemma}[Theorem A.1.12 \cite{prob}]\label{thm:prob}
Suppose there are $n$ independent trials, each with probability of success $p$. Let $Z$ be the number of successes.
For $\beta>1$, we have $$ \mathbb{P}[Z \geq \beta p n]<[e^{\beta-1} \beta^{-\beta}]^{pn}.$$
\end{lemma}
Let $Z_i$ count the number of elements in $J^{i}(A^{-}(x))$ which are included in $S$. 
Using Lemma~\ref{thm:prob} with $\beta=1/(10p)$, we have 
\begin{align*}
\mathbb{P}[ Z_i \geq i/10] & = \mathbb{P}[ Z_i \geq \beta pi] < \left(e^{\frac{1}{10p}-1} \left(\frac{1}{10p}\right)^{-\frac{1}{10p}}\right)^{pi} < (10ep)^{i/10}.
\end{align*}
Let $U_i$ count the number of the first $i/2$ elements in $N^{+}(x) \cap N^{-}(z_i)$ which are in $S$.
A near identical calculation to the above yields $\mathbb{P}[U_i \geq i/20] <(10ep)^{i/20}$. 
We have identical calculations to the above for the events of being out-bad.

Overall, note that item (1) from the definitions of in-bad and out-bad is independent of items (2) and (3), for both being in-bad and out-bad.
Therefore the probability that $x$ is in-bad or out-bad is at most
\begin{align*}
2p \left( \sum_{i=1}^{\infty} \left( (10ep)^{i/10} + (10ep)^{i/20} \right) \right),
\end{align*}
as required.
\end{proof}
\end{proof}

%For specific values of $k$ we can do better using a recursively generated tournament.
%\begin{observation}
%If $k=4 \cdot 3^i + 1$ for some integer $i \geq 0$, then there exists a tournament $T_i$ on $\frac{9}{4} (k-1)$ vertices which does not contain a strong $1$-immersion of a transitive tournament on $k$ vertices.
%\end{observation}
%The tournaments we consider are formed by recursively blowing up a directed triangle, i.e. we set $T_{-1}$ to be a directed triangle, 
%and for each $i \geq 0$, we obtain $T_i$ by blowing up each vertex of $T_{i-1}$ by $3$, and placing a directed triangle inside each blowup.
%Equivalently $T_i$ is the union of $3$ copies $A,B,C$ of $T_{i-1}$, with all edges going from $A$ to $B$, from $B$ to $C$ and from $C$ to $A$.
%First note that $T_0$ (which has $9$ vertices) does not contain a strong $1$-immersion of a transitive tournament with $5$ vertices. 
%Indeed, let $A,B,C$ be the corresponding copies of $T_{-1}$ in $T_0$. Suppose for contradiction $T_0$ contains a $1$-immersion of a transitive tournament on $k$ vertices. Let $S\subset T_0$ be the branch vertices. Clearly, we can not have $A\subset S$ and similarly for $B$ and $C$. Hence, by the pigeonhole principle there are two parts which contain exactly two vertices of $S$, say $A$ and $B$. 
%Then using induction, for each $T_i$ with $i \geq 1$, we can only select at most $4 \cdot 3^{i-1}$ vertices from each copy of $T_{i-1}$, so overall the largest strong $1$-immersion of a transitive tournament we can hope to find has $4 \cdot 3^{i}$ vertices. 

\section{Proof of Theorem~\ref{thm:immersion2}}\label{sec:proof2}

First we need two simple lemmas on tournaments.

\begin{lemma}
Let $T$ be a tournament on $n$ vertices and let $0<\varepsilon<1$. 
Then $T$ contains at least $n(1-\varepsilon)$ vertices with out-degree and in-degree at least $\varepsilon n/4$. 
\end{lemma}

\begin{proof}
Pick an ordering $x_1,\dots,x_n$, with smallest number of backward edges, i.e. those of the form $x_ix_j$ with $i>j$. 
We have $d^{-}(x_i) \geq \lfloor i/2 \rfloor$ and $d^{+}(x_i) \geq \lfloor (n+1-i)/2 \rfloor$ for all $i \in [n]$. 
Now the middle $(1-\varepsilon)n$ vertices have out-degree and in-degree at least $\varepsilon n/4$ as required. 
\end{proof}

\begin{lemma}\label{lem:differ}
Let $T$ be a tournament on $n$ vertices and let $0<\varepsilon<1$. 
Then, for every $t \leq n$, $T$ contains $(1-\varepsilon)t$ vertices with minimum out-degree and minimum in-degree at least $\varepsilon n/4$ and whose out-degrees differ by at most $t$.
\end{lemma}

\begin{proof}
Use the above lemma to obtain $n(1-\varepsilon)$ vertices with out-degree and in-degree at least $\varepsilon n/4$. 
Now, split the set $[n]$ into intervals of size $t$. 
Put each of the above vertices into the interval which contains its out-degree.
There are at most $n/t$ intervals and so by the pigeonhole principle one interval contains at least $(1-\varepsilon)t$ vertices. 
\end{proof}

\begin{proof}[Proof of Theorem~\ref{thm:immersion2}]
We let $C:=59$. We will repeatedly do the following process, until we are able to find an immersion as desired. 
Initialise by setting $i:=1$ and $T_1:=T$. 

We then repeat the following.

{\bf Step $i$:}

First set $R_i$ to be a set of $k$ vertices in $T_i$ which each have out-degree and in-degree at least $4k$ within $T_i$, and whose out-degrees differ by at most $2k$ within $T_i$.
We know this exists by applying Lemma~\ref{lem:differ} with $n\geq 32k$, $\varepsilon=1/2$ and $t=2k$. 
Since we use $C=59$, it is clear that $|V(T_1)| \geq 32k$; for $T_i$ with $i \geq 2$, we explain later.
We try to find an immersion inside $T_i$, using $R_i$ as the branch vertices.

If we succeed, then we are done, so suppose not.
There exists two vertices $x_i,y_i \in R_i$ for which we cannot find a directed path of length $3$ from $x_i$ to $y_i$, which extends the current attempt to find an immersion.
In particular, we may have managed to construct paths of length $3$ from $x_i$ to each other vertex in $R_i$, and from each other vertex in $R_i$ to $y_i$.
Suppose that these paths go through sets $X_i$ and $Y_i$ respectively, and are of minimal length. 
That is, define $X_i$ to be the set of vertices which are contained in the constructed paths between $x_i$ and vertices in $R_i$ of the immersion being built. Define $Y_i$ similarly (excluding those vertices already in $X_i$) for paths between vertices in $R_i$ and $y_i$.
By minimality, we mean that for each constructed path, it is not possible to shortcut the path by removing any internal vertices within the path.
Let $A_i$ be the out-neighbours of $x_i$ outside of $R_i \cup X_i \cup Y_i$, and let $B_i$ be the set of in-neighbours
of $y_i$ outside of $R_i \cup X_i \cup Y_i$. 
Note that $R_i,A_i,B_i,X_i,Y_i \subseteq V(T_i)$; let $C_i$ be all of the remaining vertices in $T_i$.
See Figure~\ref{fig:lookforimmersion} for a sketch.

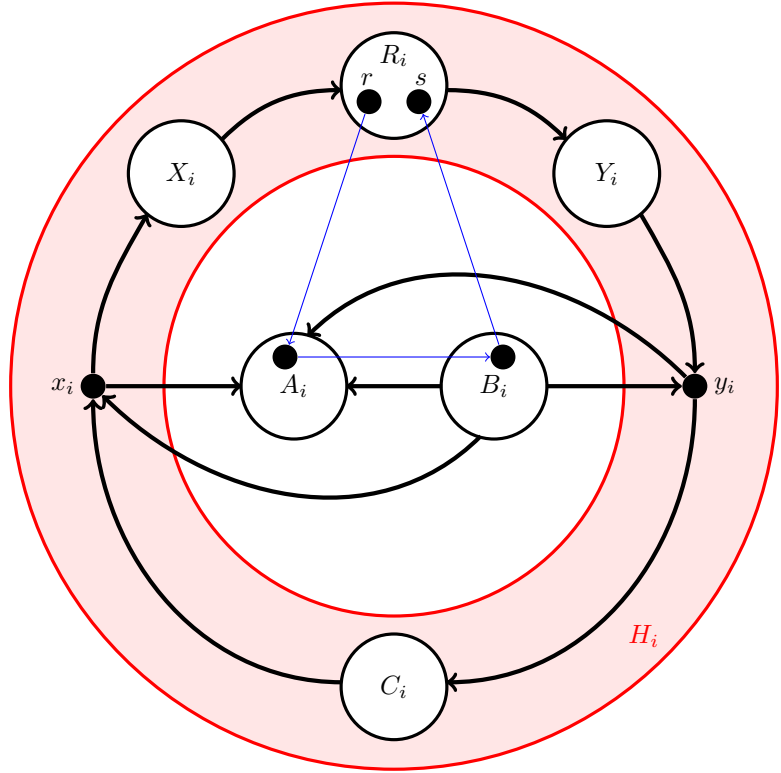
\begin{figure}[h]
\begin{center}
\begin{tikzpicture}[scale=1.0]

\filldraw[color=red!100, fill=red!10, very thick] (0:0) circle(145pt);
\filldraw[color=red!100, fill=red!0, very thick] (0:0) circle(87pt);

\node[vertex] at (180:4) (x) {};
\node[vertex] at (0:4) (y) {};

\filldraw[color=black!100, fill=red!0, very thick] (90:4) circle(20pt);
\filldraw[color=black!100, fill=red!0, very thick] (180:1.33) circle (20pt);
\filldraw[color=black!100, fill=red!0, very thick] (0:1.33) circle (20pt);
\filldraw[color=black!100, fill=red!0, very thick] (45:4) circle (20pt);
\filldraw[color=black!100, fill=red!0, very thick] (135:4) circle (20pt);
\filldraw[color=black!100, fill=red!0, very thick] (270:4) circle (20pt);

\node[vertex] at (85:3.8) (s) {};
\node[vertex] at (95:3.8) (r) {};
\node[vertex] at (15:1.5) (b) {};
\node[vertex] at (165:1.5) (a) {};

\draw[->,ultra thick] (x) -- (180:2.02);
\draw[<-,ultra thick] (y) -- (0:2.02);
\draw[<-,ultra thick] (180:0.64) -- (0:0.64);

\draw[<-,ultra thick] (x) to[out=-45,in=-135] (330:1.33);
\draw[->,ultra thick] (y) to[out=135,in=45] (150:1.33);

\draw[->,ultra thick] (x) to[out=90,in=-120] (145:4);
\draw[->,ultra thick] (125:4) to[out=45,in=180] (100:4);
\draw[->,ultra thick] (80:4) to[out=0,in=135] (55:4);
\draw[->,ultra thick] (35:4) to[out=-60,in=90] (y);
\draw[<-,ultra thick] (x) to[out=270,in=180] (260:4);
\draw[->,ultra thick] (y) to[out=270,in=0] (280:4);

\draw[->,blue] (r) -- (a);
\draw[->,blue] (a) -- (b);
\draw[->,blue] (b) -- (s);

\node at (95:4.1) {$r$};
\node at (85:4.1) {$s$};
\node at (180:4.4) {$x_i$};
\node at (0:4.4) {$y_i$};
\node at (180:1.33) {$A_i$};
\node at (0:1.33) {$B_i$};
\node at (45:4) {$Y_i$};
\node at (90:4.4) {$R_i$};
\node at (135:4) {$X_i$};
\node at (270:4) {$C_i$};
\node at (315:4.7) {{\color{red}$H_i$}};

\end{tikzpicture}
\end{center}
\caption{A sketch of $x_i, y_i$, the sets $R_i,A_i,B_i,X_i,Y_i$ and edges between them. The set $H_i$ contains everything except $A_i$ and $B_i$. If we assume there is no way of choosing a directed path of length $3$ going from $x_i$ to $y_i$, then every edge between $A_i$ and $B_i$ goes from $B_i$ to $A_i$, except for possibly some edges which are already used in paths from another branch vertex to another branch vertex, for example the blue path between $r$ and $s$ shown here. Such an edge cannot be used in a path between $x_i$ and $y_i$, since paths between branch vertices must be edge-disjoint from each other.} 
\label{fig:lookforimmersion}
\end{figure}

Now let $H_i:=V(T_i) \setminus (A_i \cup B_i)$ and $D_i:=A_i \cup H_i$. 

Finally we claim that either:
\begin{itemize}
\item[(P1)] All but at most $k-1$ vertices in $D_i$ have out-degree at least $(C-1)k$ in $D_i$.
If so, define $T_{i+1}:=T_i[D_i]$ and move onto step $i+1$; or
\item[(P2)] We can find an immersion of a complete directed graph on $k$ vertices where each path has at most $3$ edges.
\end{itemize}

If (P1) holds, then there exists a set of size $|D_i|-k$ inside $D_i$ which induces a subtournament with out-degree at least $(C-2)k$. 
Thus $|V(T_{i+1})| \geq (2C-4)k \geq 32k$ by our choice of $C$, and so we can apply Lemma~\ref{lem:differ} in step $i+1$.

Now we show why the above process must terminate.
Suppose for a contradiction that it does not.
Since we terminate if we find an immersion using branch vertices $R_i$ or if (P2) is true, we must have (P1) for all $i$, infinitely. 
In particular, combined with the previous argument this implies we have $|D_i| \geq 32k$ for all $i$.
For all $i \geq 2$ we have $D_{i-1} = D_i \cup B_i$.
If $B_i$ is non-empty for all $i$, then we have that $D_1 \supsetneq D_2 \supsetneq D_3 \supsetneq \dots$, and so this would immediately imply that there exists $j$ such that $|D_j| < 32k$, a contradiction. The fact that $B_i$ is non-empty follows quickly via $y_i$ having in-degree at least $4k$ in $T_i$;
the only other in-neighbours of $y_i$ can be at most $k-2$ from $R_i \cup Y_i$ (there are at most $k-2$ by minimality of the path lengths) and at most $2k-4$ from $X_i$. 
(Vertices in $A_i$ cannot be in-neighbours of $y_i$ as then we could find a path of length $2$ from $x_i$ to $y_i$)

The above argument shows that $|B_i| \geq k$, and a symmetrical argument shows that $|A_i| \geq k$ as well, a fact which we will use later.

Now it suffices to prove that if (P1) is not true, then (P2) holds.
Suppose that step $i$ is the first step where (P1) does not hold.
Then there exist $k$ vertices in $D_i$ which have out-degree at most $(C-1)k$ in $D_i$. 
Let these vertices be $V_i:=\{v_1,\dots,v_k\}$.
We will use these vertices as branch vertices of the immersion. 
First we make the following claim, which will allow us to easily choose the paths.
\begin{claim}
For each $j \in [i]$ and $t \in [k]$, we have
\begin{itemize}
\item[(i)] $v_t$ has out-degree at least $k-1$ in $\cup_{\ell \in [i]} B_{\ell}$;
\item[(ii)] $v_t$ has in-degree at least $3k-1$ in $A_j \setminus V_i$.
\end{itemize}
\end{claim}

\begin{proof}
For (i), since $v_t$ has out-degree at least $Ck$ in $T$ and out-degree at most $(C-1)k$ in $D_i$, it has out-degree at least $k$ in $V(T) \setminus D_i = \cup_{\ell \in [i]} B_{\ell}$.

For (ii), first we bound the size of $|C_j|$ and hence $|H_j|$ for each $j \in [i]$.
Since we cannot find a directed path of length $3$ from $x_j$ to $y_j$ outside of $R_j \cup X_j \cup Y_j$,
every edge between $A_j$ and $B_j$ goes from $B_j$ to $A_j$,
except for possibly some edges which are already used in a directed path of length $3$ from one branch vertex to another branch vertex.
Additionally, all vertices in $A_j$ are out-neighbours of $y_j$ and 
all vertices in $B_j$ are in-neighbours of $x_j$.
The vertices in $C_j$ are not out-neighbours of $x_j$ (otherwise they would be in $A_j$), and are not in-neighbours of $y_j$ (otherwise they would be in $B_j$).
Again, refer to Figure~\ref{fig:lookforimmersion} for an illustration.
Now, $x_j$ and $y_j$ have the exact same out-degree within $A_j \cup B_j$. 
We have $x_j$ has out-degree at most $k-2$ in $R_j \cup X_j$, and has out-degree at most $2k$ in $Y_j$.  
Overall, $x_j$ has out-degree at most $3k$ larger than the out-degree of $y_j$ so far.  
That is, 
$$|(N^{+}_{T_j}(x_j))\cap(A_j \cup B_j \cup R_j \cup X_j \cup Y_j)| - |(N^{+}_{T_j}(y_j))\cap(A_j \cup B_j \cup R_j \cup X_j \cup Y_j)| \leq 3k.$$
The vertices in $C_j$ contribute only to the out-degree of $y_j$ and not of $x_j$,
(i.e. $|(N^{+}_{T_j}(x_j))\cap C_j|=0$ and $|(N^{+}_{T_j}(y_j))\cap C_j|=|C_j|$)
and so since the out-degree of $y_j$ is at most $2k$ larger than the out-degree of $x_j$, we obtain 
\begin{align*}
|C_j| = & \, |(N^{+}_{T_j}(y_j))\cap C_j| 
= d^{+}_{T_j}(y_j) - |(N^{+}_{T_j}(y_j))\cap(A_j \cup B_j \cup R_j \cup X_j \cup Y_j)| \\
\leq & \, d^{+}_{T_j}(x_j)+2k - |(N^{+}_{T_j}(y_j))\cap(A_j \cup B_j \cup R_j \cup X_j \cup Y_j)| \\ 
\leq & \, d^{+}_{T_j}(x_j)+2k - |(N^{+}_{T_j}(x_j))\cap(A_j \cup B_j \cup R_j \cup X_j \cup Y_j)| + 3k = 5k.
\end{align*}
Recalling that $|X_j|, |Y_j| \leq 2k-4$ and $|R_j| \leq k-2$,
this further implies that $|H_j| \leq 10k$.

Recall that $v_t$ has out-degree at most $(C-1)k$ in $D_i$.
This means it has in-degree at least $|D_i|-(C-1)k-1$ in $D_i$.
We have $D_i \subseteq D_j = A_j \cup H_j$.
It follows that $v_t$ has in-degree at least $|D_i|-(C-1)k-1$ in $A_j \cup H_j$.
Since $|H_j| \leq 10k$, $v_t$ has in-degree at least 
$|D_i|-(C+9)k-1$ in $A_j$,
which in turn implies it has in-degree at least $|D_i|-(C+10)k-1$ in $A_j \setminus V_i$.
Therefore we are done if we can show that $|D_i| \geq (C+13)k$.

Now recall that there are no edges going from $A_j$ to $B_j$ for any $j \in [i]$ (except for those already used in paths between other branch vertices).
Therefore if $i=1$, then all vertices in $A_i$ have out-degree at least $(C-1)k$ in $D_i$.
If $i>1$, then since (P1) holds for $i-1$, all but at most $k-1$ vertices in $D_{i-1}$ have out-degree at least $(C-1)k$ in $D_{i-1}$.
Further, since $V(T_i)=D_{i-1}$ and $|A_i| \geq k$, there is at least one vertex in $A_i$ with out-degree at least $(C-2)k$ in $D_i$.
First of all, this implies $|D_i| \geq (C-2)k$.
We will obtain an even better bound.

Again using that $|H_i| \leq 10k$, we have $|A_i| \geq (C-12)k$,
and further, all but at most $k-1$ of these vertices have out-degree at least $(C-2)k$ in $D_i$.
It follows that 
\begin{align}
|D_i|^2 \geq |D_i| (|D_i|-1) = 2 \sum_{x \in D_i} d^+_{D_i}(x) \geq 2(C-13)(C-2) k^2,
\end{align}
from which it follows that $|D_i| \geq (\sqrt{2(C-13)(C-2)})k$. 
Therefore we require $\sqrt{2(C-13)(C-2)} \geq C+13$, which is true for $C \geq 59$, completing the proof of the claim.
\end{proof}

We now want to find directed paths from $v_r$ to $v_s$ for each ordered pair $(r,s)$ with $r,s \in [k]$, $r \not=s$, which satisfy the properties of an immersion of a complete directed graph on $k$ vertices with paths of length at most $3$.
That is, for each ordered pair $(r,s)$ we must find a directed path of length at most $3$ from $v_r$ to $v_s$, such that the chosen directed paths for each pair are pairwise edge-disjoint.

Given a pair $(r,s)$, we construct the path from $v_r$ to $v_s$ as follows.
Since $v_r$ has out-degree at least $k-1$ in $\cup_{\ell \in [i]} B_i$, we may choose one such out-neighbour $b_{rs}$ which is not yet an out-neighbour of $v_r$ used for the path for $(r,s')$ for some $s' \not=s$. 
Suppose this vertex is in $B_t$, $t \in [i]$.
Now since $v_s$ has in-degree at least $3k-1$ in $A_t \setminus V_i$, we may choose one such in-neighbour $a_{rs}$ which is 
\begin{itemize}
\item not yet an in-neighbour of $v_s$ used for the path for $(r',s)$ for some $r' \not=r$; 
\item not already used together with $b_{rs}$ in a directed path of length $3$ between other branch vertices;
\item is an out-neighbour of $b_{rs}$ (recall that vertices in $B_t$ have all but at most $k-1$ vertices in $A_t$ as out-neighbours).
\end{itemize}
It is now easy to see that the constructed paths $\{v_r b_{rs} a_{rs} v_s : r,s \in [k], r\not=s\}$ are edge-disjoint from each other as required. 
Further note that by construction,  branch vertices are not used as the midpoints of any of these paths, so this is indeed a strong immersion.
\end{proof}

\section{Concluding remarks}\label{sec:conclusion}
We conclude our paper with some open questions. The first one asks for the precise order of a largest tournament without containing an immersion of a transitive tournament on $k$ vertices. Even though our arguments could be optimised to give a better multiplicative constant it is very unlikely they could be pushed to obtain the right constant. 
\begin{conjecture}
    Every tournament on at least $(2+o(1))k$ vertices contains a $1$-immersion of a transitive tournament on $k$ vertices. 
\end{conjecture}

As stated in the introduction, Lochet proved that for every positive integer $k$, there exists $f(k)$ such that
any digraph with minimum out-degree at least $f(k)$ contains a immersion of a transitive tournament on $k$
vertices. Recently, the first author and Letzter~\cite{LetzterGirao} showed that every eulerian digraph with minimum out-degree at least $Ck$ must contain an immersion of a complete digraph on $k$ vertices. We believe some linear bound should also be enough to find an immersion in arbitrary digraphs. 

We restate here a conjecture which appeared originally in~\cite{LetzterGirao}.

\begin{conjecture}
    There exists an absolute constant $C>0$ such that for any positive integer $k$ the following holds. Let $D$ be a digraph with $\delta^{+}(D)\geq Ck$. Then $D$ contains an immersion of a transitive tournament on $k$ vertices.
\end{conjecture}

\section*{Acknowledgements}
The authors would like to thank the two anonymous referees for their careful and helpful reviews.

\bibliographystyle{amsplain}
\bibliography{immersions_tournaments.bib}

\providecommand{\bysame}{\leavevmode\hbox to3em{\hrulefill}\thinspace}
\providecommand{\MR}{\relax\ifhmode\unskip\space\fi MR }
% \MRhref is called by the amsart/book/proc definition of \MR.
\providecommand{\MRhref}[2]{%
  \href{http://www.ams.org/mathscinet-getitem?mr=#1}{#2}
}
\providecommand{\href}[2]{#2}
\begin{thebibliography}{10}

\bibitem{Alon}
N.~Alon, \emph{Subdivided graphs have linear {R}amsey numbers}, J. Graph Theory
  \textbf{18} (1994), no.~4, 343--347.

\bibitem{AlonKrivSudak}
N.~Alon, M.~Krivelevich, and B.~Sudakov, \emph{{T}ur{\'a}n numbers of bipartite
  graphs and related {R}amsey-type questions}, Combin. Probab. Comput.
  \textbf{12} (2003), 477--494.

\bibitem{prob}
N.~Alon and J.~H. Spencer, \emph{The {P}robabilistic {M}ethod}, John Wiley \&
  Sons, 2016.

\bibitem{BollobasThomason}
B.~Bollob{\'a}s and A.~Thomason, \emph{Highly linked graphs}, Combinatorica
  \textbf{16} (1996), no.~3, 313--320.

\bibitem{burr1975magnitude}
S.~A. Burr and P.~Erd\H{o}s, \emph{On the magnitude of generalized {R}amsey
  numbers for graphs, {I}nfinite and finite sets ({C}olloq., {K}eszthely, 1973;
  dedicated to {P}. {E}rd\H{o}s on his 60th birthday), vol. 1}, Colloq. Math.
  Soc, vol.~1, 1975, pp.~214--240.

\bibitem{Foxetal}
M.~DeVos, Z.~Dvo\v{r}\'ak, J.~Fox, J.~McDonald, B.~Mohar, and D.~Scheide,
  \emph{Minimum degree condition forcing complete graph immersion},
  Combinatorica \textbf{34} (2014), 279--298.

\bibitem{Draganicetal}
N.~Dragani\'c, D.~Munh\'a~Correia, B.~Sudakov, and R.~Yuster, \emph{Ramsey
  number of 1-subdivisions of transitive tournaments}, J. Combin. Theory,
  Series {B} \textbf{157} (2022), 176–183.

\bibitem{LetzterGirao}
A.~Gir\~ao and S.~Letzter, \emph{Immersion of complete digraphs in eulerian
  digraphs}, Israel J. Mathematics (In Press) (2023).

\bibitem{Popielarz}
A.~Gir\~ao, K.~Popielarz, and R.~Snyder, \emph{Subdivisions of digraphs in
  tournaments}, J. Combin. Theory, Series {B} \textbf{146} (2021), 266--285.

\bibitem{KomlosSzemeredi}
J.~Koml\'os and E.~Szemer\'edi, \emph{Topological cliques in graphs {II}},
  Combin. Probab. and Comput. \textbf{5} (1996), 79--90.

\bibitem{hongetal}
H.~Liu, G.~Wang, and D.~Yang, \emph{Clique immersion in graphs without a fixed
  bipartite graph}, J. Combin. Theory, Series {B} \textbf{157} (2022),
  346--365.

\bibitem{Lochet}
W.~Lochet, \emph{Immersion of transitive tournaments in digraphs with large
  minimum outdegree}, J. Combin. Theory, Series {B} \textbf{134} (2019),
  350--353.

\bibitem{Mader3}
W.~Mader, \emph{Degree and local connectivity in digraphs}, Combinatorica
  \textbf{5} (1985), no.~2, 161--165.

\bibitem{Thomasseneven}
C.~Thomassen, \emph{Even cycles in directed graphs}, European J. Combin.
  \textbf{6} (1985), no.~1, 85--89.

\end{thebibliography}
\end{document}